\documentclass[12pt,reqno]{amsart}

\usepackage[foot]{amsaddr}
\usepackage{amsfonts}
\usepackage{amsmath}
\usepackage{amssymb}
\usepackage{amsthm}
\usepackage[english]{babel}
\usepackage{bm}
\usepackage{booktabs}
\usepackage{colortbl}
\usepackage{diagbox}
\usepackage{epsfig}
\usepackage{epstopdf}
\usepackage{enumerate}
\usepackage{enumitem}
\usepackage[T1]{fontenc}
\usepackage{framed}
\usepackage{fullpage}
\usepackage[left=2.50cm, right=2.50cm, top=2.50cm, bottom=2.50cm]{geometry}
\usepackage{graphicx}
\usepackage{hyperref}
\usepackage{lastpage}
\usepackage{latexsym}
\usepackage{mathtools}
\usepackage{mathrsfs}
\usepackage{multirow}
\usepackage{pstricks}
\usepackage{subfig}
\usepackage{tikz}
\usepackage{tikz-cd}
\usepackage{tikz-qtree}
\usepackage{todonotes}
\usepackage{url}
\usepackage{verbatim}
\usepackage{xcolor}

\usetikzlibrary{arrows}
\usetikzlibrary{automata}
\usetikzlibrary{cd}
\usetikzlibrary{matrix}
\usetikzlibrary{positioning}
\usetikzlibrary{patterns}

\usepackage{cleveref}
\hypersetup{hypertex=true,colorlinks=true,
linkcolor=red,
menucolor=cyan,
citecolor=cyan,
filecolor=cyan,
urlcolor=blue
}

\newcommand{\mbfa}{\mathbf{a}}
\newcommand{\mbfb}{\mathbf{b}}
\newcommand{\mbfc}{\mathbf{c}}

\newcommand{\mbz}{\mathbb{Z}}

\newcommand{\mbn}{\mathbb{N}}

\newcommand{\ol}{\overline}

\makeatletter
\newcommand{\proofpart}[2]{
  \par
  \addvspace{\medskipamount}
  \noindent\textbf{Part #1. #2}
  \par\nobreak\smallskip
  \@afterheading
}
\makeatother

\DeclareMathOperator{\conv}{conv}

\DeclareMathOperator{\Tr}{Tr}

\def\centerarc[#1](#2)(#3:#4:#5)
{ \draw[#1] ($(#2)+({#5*cos(#3)},{#5*sin(#3)})$) arc (#3:#4:#5); }

\makeatletter
\renewcommand{\email}[2][]{%
  \ifx\emails\@empty\relax\else{\g@addto@macro\emails{,\space}}\fi%
  \@ifnotempty{#1}{\g@addto@macro\emails{\textrm{(#1)}\space}}%
  \g@addto@macro\emails{#2}%
}
\makeatother

\makeatletter

\@namedef{subjclassname@2010}{

  \textup{2020} Mathematics Subject Classification}

\makeatother

\newtheorem{thm}{Theorem}[section]

\newtheorem*{thm*}{Theorem}
\newtheorem{conj}[thm]{Conjecture}
\newtheorem{cor}[thm]{Corollary}
\newtheorem{prop}[thm]{Proposition}
\newtheorem{prob}[thm]{Problem}
\newtheorem{lem}[thm]{Lemma}
\theoremstyle{definition}

\newtheorem{rmk}{Remark}
\newtheorem{exa}[thm]{Example}

\numberwithin{equation}{section}

\newcommand{\newabstract}[1]{%
  \par\bigskip
  \csname otherlanguage*\endcsname{#1}%
  \csname captions#1\endcsname
  \item[\hskip\labelsep\scshape\abstractname.]
}
\makeatletter
\tikzset{
        hatch distance/.store in=\hatchdistance,
        hatch distance=5pt,
        hatch thickness/.store in=\hatchthickness,
        hatch thickness=5pt
        }
\pgfdeclarepatternformonly[\hatchdistance,\hatchthickness]{north east hatch}
    {\pgfqpoint{-1pt}{-1pt}}
    {\pgfqpoint{\hatchdistance}{\hatchdistance}}
    {\pgfpoint{\hatchdistance-1pt}{\hatchdistance-1pt}}%
    {
        \pgfsetcolor{\tikz@pattern@color}
        \pgfsetlinewidth{\hatchthickness}
        \pgfpathmoveto{\pgfqpoint{0pt}{0pt}}
        \pgfpathlineto{\pgfqpoint{\hatchdistance}{\hatchdistance}}
        \pgfusepath{stroke}
    }
\makeatother

\begin{document}

\title[New arithmetic invariants for cospectral graphs]{New arithmetic invariants for cospectral graphs}

\author{Yizhe Ji}
\author{Quanyu Tang}
\author{Wei Wang}
\author{Hao Zhang}
\address[Yizhe Ji, Quanyu Tang, Wei Wang]{School of Mathematics and Statistics, Xi'an Jiaotong University, Xi'an 710049, P. R. China}
\address[Hao Zhang]{School of Mathematics, Hunan University, Changsha 410082, P. R. China}
\email[Yizhe Ji]{jyz67188310@stu.xjtu.edu.cn}
\email[Quanyu Tang]{tang\_quanyu@163.com}
\email[Wei Wang]{wang\_weiw@xjtu.edu.cn}
\email[Hao Zhang (Corresponding author) ]{zhanghaomath@hnu.edu.cn}

\date{}

\begin{abstract}
    An \emph{invariant for cospectral graphs} is a property shared by all cospectral graphs. In this paper, we establish three novel arithmetic invariants for cospectral graphs, revealing deep connections between spectral properties and combinatorial structures. More precisely, one of our main results shows that for any two cospectral graphs $G$ and $H$ with adjacency matrices $A(G)$ and $A(H)$, respectively, the following congruence holds for all integers $m\geq 0$:
    \[e^{\rm T}A(G)^me\equiv e^{\rm T}A(H)^me \pmod{4},\]
where $e$ is the all-one vector. Moreover, we present a number of fascinating applications. Specifically: i) Resolving a conjecture proposed by the third author, we demonstrate that under certain conditions, every graph cospectral with a graph $G$ is determined by its generalized spectrum. ii) We demonstrate that whenever the complements of two trees are cospectral, then one tree has a perfect matching if and only if the other does. An analogous result holds for the existence of triangles in general graphs. iii) An unexpected connection to the polynomial reconstruction problem is also provided, showing that the parity of the constant term of the characteristic polynomial is reconstructible.
\end{abstract}

\subjclass[2020]{Primary 05C50.}

\maketitle

\noindent\textbf{Keywords:} Graph spectra; Cospectral graphs; Determined by generalized spectrum; Graph invariants; Closed walks

\section{Introduction}

The \emph{spectrum} of a graph, consisting of all the eigenvalues (including multiplicity) of its adjacency matrix, is a powerful tool in addressing various, often seemingly unrelated, problems in graph theory. Notable examples include Lov\'{a}sz’s groundbreaking work~\cite{Lov} in determining the Shannon capacity of the pentagon $C_5$ via a novel spectral approach, and Huang’s recent resolution of the Sensitivity Conjecture~\cite{Huang} using the spectrum of signed hypercubes.

Two graphs are \emph{cospectral} if they share the same spectrum. The existence of non-isomorphic cospectral graphs was first demonstrated by Collatz and Sinogowitz~\cite{CS1} in 1957. Another well-known result in this area is Schwenk’s theorem~\cite{Sch}, which states that almost every tree has a cospectral mate. For more constructions of cospectral graphs, see, e.g.~\cite{LS,GM,S}. On the other side of the coin, the question ``Which graphs are determined by their spectrum (DS for short)?" was first raised in 1956 by G\"{u}nthard and Primas~\cite{GH}, relating H\"{u}ckle's theory in chemistry to graph spectra. Proving that a given graph is DS is generally difficult and challenging. To date, only a few graphs with very specific structures are known to be DS. For further background and results, we refer the reader to~\cite{DH1, DH2}.

A \emph{graph invariant} is a property that remains unchanged for all isomorphic graphs. Examples include the number of vertices, edges, the degree sequence, the clique number, and the chromatic number, among others. Given two graphs $G$ and $H$, if one can find a graph invariant for which the two graphs are different, then clearly $G$ and $H$ are non-isomorphic. Therefore, graph invariants are useful in showing that two graphs are non-isomorphic.

In this paper, we investigate graph invariants for cospectral graphs, i.e. properties shared by all cospectral graphs. It is well-known that the number of vertices, edges, triangles, the bipartiteness and the regularity of a graph are invariants for cospectral graphs. However, properties such as the degree sequence, clique number, and chromatic number are not generally preserved by cospectrality.

Identifying invariants for cospectral graphs provides deeper insights into their structural properties. Nevertheless, except for the aforementioned ones, very few invariants for cospectral graphs are known in the literature.  A natural problem is
\begin{prob}\label{Prob1}
Can one give more graph invariants for cospectral graphs?
\end{prob}

The main objective of the paper is to give some graph invariants for cospectal graphs, which, to the best of our knowledge, have not been previously studied.

\subsection{Main results}

Our personal interest in Problem~\ref{Prob1} comes from the generalized spectral characterizations of graphs. Let $G=(V,E)$ be a simple graph with vertex set $V=\{v_1,\dots, v_n\}$ and edge set $E$. The \emph{adjacency matrix} of $G$ is an $n$ by $n$ matrix $A=A(G)=(a_{ij})$, where $a_{ij}=1$ if $v_i$ and $v_j$ are adjacent, and $a_{ij}=0$ otherwise. The \emph{characteristic polynomial} of $G$ is defined as $\phi(G;x)=\det(xI-A(G))$. A graph $G$ is said to be \emph{determined by its generalized spectrum} (DGS for short) if,  whenever $H$ is a graph such that $H$ and $G$ are cospectral with cospectral complements, then $H$ is isomorphic to $G$.

 Let
\[W(G)=[e,A(G)e,\ldots,A(G)^{n-1}e],\]
be the \emph{walk-matrix} of $G$, where $e$ is the all-one vector. Define $\eta(G):=2^{-\lfloor n/2 \rfloor}|\det W(G)|$, which is always an integer according to~\cite{WX1}. The invariant $\eta(G)$ plays an important role in showing a graph $G$ to be DGS. In~\cite{Wang2}, Wang proved the following

\begin{thm}[\cite{Wang2}, Theorem 1.1] If $\eta(G)$ is odd and square-free, then $G$ is DGS.
\end{thm}

 Some time ago, the third author of this paper proposed the following

\begin{conj}\label{parity}
 Let $G$ and $H$ be two cospectral graphs. Then $\eta(G)$ and $\eta(H)$ have the same parity, i.e. $\eta(G)\equiv~\eta(H)\pmod{2}$.
\end{conj}

In this paper, we show the above conjecture is true, therefore we have the following

\begin{thm}\label{parity11}
Conjecture~\ref{parity} is true, and thus the parity of $\eta(G)$ is an invariant for cospectral graphs.
\end{thm}

Based on Theorem~\ref{parity11}, we are able to show that under certain conditions, every graph cospectral with $G$ is DGS; see Corollary~\ref{cor1} for more details. The key for proving Theorem~\ref{parity11} is the following theorem, which gives another new invariant for cospectral graphs.

\begin{thm}\label{equiv1}
Let $G$ and $H$ be two cospectral graphs with adjacency matrices $A(G)$ and $A(H)$, respectively. Then we have $e^{\rm T}A(G)^me\equiv e^{\rm T}A(H)^me\pmod{4}$ for any integer $m\geq 0$.
\end{thm}
To state the third new invariant, consider the following question:
Given a pair of cospectral graphs $G$ and $H$, what can be said about the characteristic polynomials of their complements? In the following, we derive an equivalent formulation of Theorem \ref{equiv1}, providing an unexpected answer to this question and simultaneously revealing the third new invariant. Recall that two polynomials are \emph{congruent} modulo $m$ if the coefficients of their difference are divisible by $m$.

\begin{thm}\label{equiv3}
Let $G$ and $H$ be two cospectral graphs with complements $\bar{G}$ and $\bar{H}$, respectively. Then we have $\phi(\bar{G};x)\equiv \phi(\bar{H};x)\pmod{4}$.
\end{thm}

\subsection{Applications}

We shall provide several interesting applications of the above newly obtained invariants for cospectral graphs in what follows.

First, we give an application of the above results for the generalized spectral characterizations of graphs, which constitutes the original motivation for this work. Let $f(x)\in{\mathbb{Z}}[x]$ be a polynomial with leading coefficient $a_0$. Let $\alpha_1,\alpha_2,\ldots,\alpha_n$ be all the complex roots of $f$.
The \emph{discriminant} of $f(x)$, denoted by $\Delta(f)$, is defined as
$$a_0^{2n-2}\prod_{1\leq i<j\leq n}(\alpha_i-\alpha_j)^2.$$
For example, if $f(x)=ax^2+bx+c$, then $\Delta(f)=b^2-4ac$.
The \emph{discriminant} of a matrix $A$, denoted by $\Delta(A)$, is defined as the discriminant of its characteristic polynomial $\phi(x)=\det(xI -A)$.
 The \emph{discriminant} of a graph $G$, denoted by $\Delta(G)$, is defined as the discriminant of its adjacency matrix.

More recently, Wang et al.~\cite{WYZ} obtained, among others, the following result for DGS graphs.

\begin{thm}[\cite{WYZ}, Corollary 2]\label{Dis} Let $G$ be a graph such that
{\rm i)} $\Delta(G)$ is square-free for odd primes; and {\rm ii)} $\eta(G)$ is odd. Then $G$ is DGS.
\end{thm}

As a consequence of Theorem~\ref{Dis} and Theorem~\ref{parity11}, we have the following corollary.

\begin{cor}\label{cor1}
Let $G$ be a graph such that
{\rm i)} $\Delta(G)$ is square-free for odd primes; and {\rm ii)} $\eta(G)$ is odd.
 Then every graph that is cospectral with $G$ (including $G$ itself) is DGS.
\end{cor}

\begin{proof} Let $H$ be any graph that is cospectral with $G$. Then $\Delta(H)=\Delta(G)$ since $\phi(G)=\phi(H)$.
Moreover, By Theorem~\ref{parity11}, we have $\eta(H)\equiv~\eta(G)\pmod{2}$. Thus, $\eta(H)$ is also odd. Then by Theorem~\ref{Dis}, $H$ is DGS. \end{proof}

\begin{rmk}
Apparently, Corollary~\ref{cor1} is much stronger than Theorem~\ref{Dis}, which essentially says that whenever a graph $G$ satisfies that {\rm i)} $\Delta(G)$ is square-free for odd primes and {\rm ii)} $\eta(G)$ is odd, and suppose that $H_0:=G$ and $H_1,\ldots,H_s$ are all the cospectral mates of $G$,
then non of $\bar{H}_i$ and $\bar{H}_j$ are cospectral for $i\neq j$. Of course this is trivial when $s\leq 1$, but it seems hard to give a direct proof without using Theorem~\ref{parity11}, for $s\geq 2$.
\end{rmk}

Next, we show that the congruence relation between the characteristic polynomials of the complements of graphs $G$ and $H$ modulo $4$ implies that their certain combinatorial invariants must share the same parity or structural similarities modulo $4$. In particular, for trees, the spectrum carries more refined structural information which leads to a direct correspondence between their perfect matchings.

\begin{cor}\label{cor:perfectmatch}
Let $G,H$ be two trees and $\phi(\bar{G};x)=\phi(\bar{H},x)$, then $G$ has a perfect matching if and only if $H$ has a perfect matching.
\end{cor}
{\begin{proof}
By Sachs' coefficients Theorem~(see \cite{CDS}), it is easy to see that the constant term of the characteristic polynomial of a tree can only be $0, 1$ or $-1$, and a tree has a perfect matching if and only if the constant term of its characteristic polynomial is either $1$ or $-1$. So the result follows directly from Theorem \ref{equiv3}.
\end{proof}

The following corollary shows that the spectral constraint for the complement of a graph $G$ also governs some local substructures such as triangles in $G$.

\begin{cor}\label{cor:triangle}
Let $G,H$ be two graphs with $\phi(\bar{G};x)=\phi(\bar{H},x)$, then the numbers of triangles of $G$ and $H$ have the same parity. In particular, if $G$ contains exactly one triangle, then $H$ also contains at least one triangle.
\end{cor}
\begin{proof}
It is well-known that the number of triangle of graph $G$ is equal to $\frac{1}{6}\Tr(A(G)^3)$. So by Theorem \ref{equiv3}, we see that $\phi(G;x)\equiv \phi(H;x)\pmod{4}$. In particular, we have $\frac{1}{6}\Tr(A(G)^3)\equiv \frac{1}{6}\Tr(A(H)^3)\pmod{2}$. This proves the desired result.
\end{proof}

\begin{rmk}
In other words, Corollary~\ref{cor:perfectmatch} says that for two trees $G$ and $H$, if one has a perfect matching and the other does not, then their complements are not cospectral. We do not aware a proof of this result without using Theorem \ref{equiv3}. A
similar remark applies to Corollary~\ref{cor:triangle}.
\end{rmk}

The above new invariants also find unexpected applications in the polynomial reconstruction problem, proposed by Cvetkovi\'c in 1973.
Let $G_i=G-v_i$ be the $i$-th \emph{vertex-deleted subgraph} of $G$. The \emph{polynomial deck} of $G$, denoted by $\mathcal{P}(G):=\{\phi(G_1),\ldots,\phi(G_n)\}$,
 is the multiset of the characteristic polynomials of its vertex-deleted subgraphs. The \emph{polynomial reconstruction problem} asks whether the characteristic polynomial of a graph $G$ with $n\geq 3$ is reconstructible from its polynomial deck $\mathcal{P}(G)$. For certain special families of graphs, the answer is known to be affirmative, including regular graphs, trees, unicyclic graphs, some bipartite graphs, and some disconnected graphs. Moreover, some graph invariants are known to be reconstructible from the polynomial deck, e.g.~the degree sequence, the length of the shortest odd cycle, the number of triangles, quadrangles, pentagons are determined by the polynomial deck.
However, the problem remains widely open, see~\cite{SS} for a recent survey.

The following proposition is well-known.
\begin{prop}
\[\phi^{\prime}(G;x)=\sum_{i=1}^n \phi(G_i;x).\]
\end{prop}

It follows that all the coefficients (except for the constant term) of $\phi(G;x)$ are determined by $\mathcal{P}(G)$. Therefore, the main focus would be whether the constant term of $\phi(G;x)$ is reconstructible.


More recently, Spier~\cite{Sp} showed, among others, the following result for the polynomial reconstruction problem.

\begin{thm} [Spier~\cite{Sp}, Theorem 8]\label{Spier}
For every graph $G$ with $n\geq 3$, the constant term of $\phi(G;x)~({\rm mod}~2)$ is reconstructible from $\mathcal{P}(G)$. Moreover,
for an even $n$, the constant term of $\phi(G;x)~({\rm mod}~4)$ is reconstructible from $\mathcal{P}(G)$.
\end{thm}
\begin{rmk}
Theorem~\ref{Spier} provides non-trivial information about the constant term of $\phi(G;x)$, the proof of which is based on the results established in this paper, as mentioned by Spier~\cite{Sp} "\emph{A key component of the approach to Theorem 8 is the recent result by Ji, Tang, Wang and Zhang $\cdots$}".
\end{rmk}

The paper is organized as follows. In Section \ref{sec2}, we present some preliminary lemmas. In Section \ref{sec3}, we provide the proof of Theorem \ref{equiv1}. Finally, Section \ref{sec4} concludes with the proofs of Theorems~\ref{parity11} and~\ref{equiv3}.

\section{Notation and Preliminary Lemmas}\label{sec2}
In this section, we introduce some notation and prove some useful lemmas, which are essential in the proof of Theorem \ref{equiv1}.

For a simple graph $G=(V,E)$ with vertex set $V=\{v_1,\dots, v_n\}$ and edge set $E$, a \emph{walk} of length $m$ in $G$ is a sequence of vertices (not necessarily distinct), denoted in this paper by an $(m
+1)$-tuple $(v_1,v_2,\dots,v_m,v_{m+1})$, such that $v_{i-1}$ and $v_i$ are adjacent for $1\leq i\leq m$. A walk is called \emph{closed} if $v_1=v_{m+1}$.

Let $\mathcal{W}(m)$ be the set of walks of length $m$, i.e. \[\mathcal{W}(m)=\{\mathbf{w}=(w_1,\dots,w_m,w_{m+1}): w_kw_{k+1}\in E\}.\] It is easy to see that $ e^{\rm T}A^m e$ is exactly the cardinality of $\mathcal{W}(m)$. Let
\[\mathcal{C}(m)=\{\mathbf{c}=(c_1,\dots,c_m,c_{m+1})\in \mathcal{W}(m)\,:c_1=c_{m+1}\},\]
be the set of closed walks of length $m$. Then it is known that $\Tr(A^m)$ equals to the cardinality of $\mathcal{C}(m)$. For any $\mbfc=(c_1,\dots,c_m,c_1)\in \mathcal{C}(m)$ and $d\in\mbz$, we define the \emph{$d$-translation} of $\mbfc$ as
\[\mbfc^{+d}=(c_{1+d},\dots,c_{m+d},c_{m+d+1})\in \mathcal{C}(m),\]
where the indices are understood modulo $m$, and the \emph{converse} of $c$ is given by
\[\mbfc^{\conv}=(c_1,c_m,\dots,c_2,c_1)\in \mathcal{C}(m).\]
The notation is clarified in the figure below, where the rightmost point is consistently chosen as the starting point, and counterclockwise is considered the positive direction. The left figure represents  $\mbfc$, the middle figure shows the \( 2 \)-translation of $\mbfc$, and the right figure depicts the converse of $\mbfc$.

\begin{center}
\setlength{\unitlength}{1mm}
\begin{picture}(100,35)
\put(-10,20){\circle{25}}
\put(40,20){\circle{25}}
\put(90,20){\circle{25}}
\put(2.5,20){\circle*{1}}
\put(0.825,26.25){\circle*{1}}
\put(-3.75,30.825){\circle*{1}}
\put(-10,34){\circle*{0.5}}
\put(-17,32.12){\circle*{0.5}}
\put(-13.62,33.52){\circle*{0.5}}
\put(0.825,13.75){\circle*{1}}
\put(-3.75,9.175){\circle*{1}}
\put(3.5,19){$c_1$}
\put(1.8,25.5){$c_2$}
\put(-3,30.5){$c_3$}
\put(1.8,12.5){$c_{m}$}
\put(-3,8){$c_{m-1}$}
\put(-10,2){$\mbfc$}
\put(52.5,20){\circle*{1}}
\put(50.825,26.25){\circle*{1}}
\put(46.25,30.825){\circle*{1}}
\put(40,34){\circle*{0.5}}
\put(33,32.12){\circle*{0.5}}
\put(36.38,33.52){\circle*{0.5}}
\put(50.825,13.75){\circle*{1}}
\put(46.25,9.175){\circle*{1}}
\put(53.5,19){$c_3$}
\put(51.8,25.5){$c_4$}
\put(47,30.5){$c_5$}
\put(51.8,12.5){$c_{2}$}
\put(47,8){$c_{1}$}
\put(40,2){$\mbfc^{+2}$}
\put(102.5,20){\circle*{1}}
\put(100.825,26.25){\circle*{1}}
\put(96.25,30.825){\circle*{1}}
\put(90,34){\circle*{0.5}}
\put(83,32.12){\circle*{0.5}}
\put(86.38,33.52){\circle*{0.5}}
\put(100.825,13.75){\circle*{1}}
\put(96.25,9.175){\circle*{1}}
\put(103.5,19){$c_1$}
\put(101.8,25.5){$c_m$}
\put(97,30.5){$c_{m-1}$}
\put(101.8,12.5){$c_{2}$}
\put(97,8){$c_{3}$}
\put(90,2){$\mbfc^{\conv}$}
\qbezier(0,20)(0,25.77)(-5,28.66)
\put(-5,28.66){\line(2,-3){1}}
\put(-5,28.66){\line(3,-1){1.7}}
\qbezier(50,20)(50,25.77)(45,28.66)
\put(45,28.66){\line(2,-3){1}}
\put(45,28.66){\line(3,-1){1.7}}
\qbezier(100,20)(100,25.77)(95,28.66)
\put(95,28.66){\line(2,-3){1}}
\put(95,28.66){\line(3,-1){1.7}}
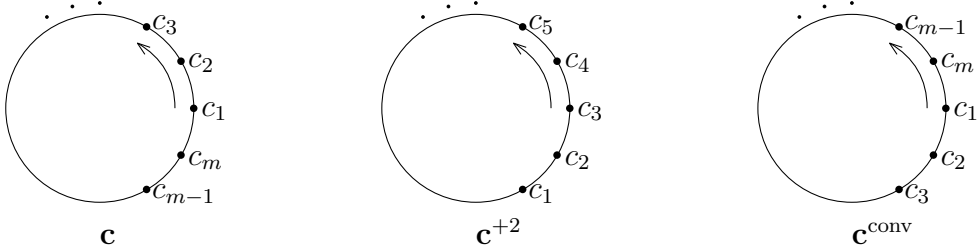
\put(20,-2){\text{The }$2$\text{-translation and converse of }$\mbfc$}
\end{picture}\\
\end{center}

It is easy to see that
\begin{equation}\label{eq:cconv}
    (\mbfc^{\conv})^{\conv}=\mbfc,\qquad (\mbfc^{+d})^{\conv}=(\mbfc^{\conv})^{+(m-d)}.
\end{equation}

For any $\mbfc\in \mathcal{C}(m)$, we define $\bar{\mbfc}:=\{\mbfa\in \mathcal{C}(m)\,:\mbfa=\mbfc^{+k} \text{ for some } k\in \mbz\}$. This is precisely the set of all translations of $\mbfc$.



The following lemma is fundamental and will be used frequently in the subsequent analysis.

\begin{lem}\label{lem:2k1}
    Let \( A \) be an integral symmetric matrix with even entries on the diagonal. Then, for any positive integer \( k \), the diagonal entries of \( A^{2k+1} \) are all even numbers.
\end{lem}

\begin{proof}
    We first assume that $A$ is an adjacency matrix of a simple graph. It is well-known that the $(i,i)$-th entry of $A^{2k+1}$ is the number of closed walks from the vertex $v_i$ to $v_i$ itself of length $2k+1$, i.e.
    \[|\{\mbfc=(c_1,\dots,c_{2k+1},c_1)\in \mathcal{C}(2k+1)\,: c_1=v_i\}|.\]
    Note that for such a walk $\mbfc$, $\mbfc^{\conv}$ is also a closed walk from $v_i$ to $v_i$ and $\mbfc$ can never equal to $\mbfc^{\conv}$. Otherwise, if $\mbfc=(c_1,\dots,c_{2k+1},c_1)$ such that $\mbfc=\mbfc^{\conv}$, then we have $c_i=c_{2k+3-i}$ for $i=2,\dots,2k+1$. In particular, we have $c_{k+1}=c_{k+2}$. However, this is impossible since $G$ is a simple graph without loops. This shows that the number of such walks is even.

    Generally, let $A=A_1+2A_2$ where $A_1$ is an adjacency matrix of a simple graph and $A_2$ is an integral matrix. Note that $A^{2k+1}\equiv A_1^{2k+1}\pmod{2}$. So it follows immediately from the special case.
\end{proof}

The next two lemmas examine the conditions under which a walk $\mbfc \in \mathcal{C}(m)$ coincides with its converse. This distinction is crucial, as walks that are not identical to their converse are counted twice, while those satisfying $\mbfc = \mbfc^{\conv}$ contribute only once.

\begin{lem}\label{lem:akconv}
    Let $\mbfc\in \mathcal{C}(m)$ with $|\bar{\mbfc}|=m$ and $4\,|\, m$, then $\bar{\mbfc}=\ol{\mbfc^{\conv}}$ if and only if $(\mbfc^{+k})^{\conv}=\mbfc^{+k}$ for some $k$. Moreover, there are exactly two choices of such $k$ in the sense of modulo $m$.
\end{lem}

\begin{proof}
    If $(\mbfc^{+k})^{\conv}=\mbfc^{+k}$ for some $k$, then for any $\ell$, by using the relation \eqref{eq:cconv}, we have
    \[\mbfc^{+\ell}=\left(\left(\mbfc^{+k}\right)^{\conv}\right)^{+(\ell-k)}=\left(\left(\mbfc^{\conv}\right)^{+(m-k)}\right)^{+(\ell-k)}=\left(\mbfc^{\conv}\right)^{+(m+\ell-2k)}\in \ol{\mbfc^{\conv}}.\]
    This shows that $\bar{\mbfc}\subseteq \ol{\mbfc^{\conv}}$. It is similar to show that $\ol{\mbfc^{\conv}}\subseteq \bar{\mbfc}$.

    Conversely, if $\bar{\mbfc}=\ol{\mbfc^{\conv}}$, then there exists an $i$ with $0\leq i<m$ such that $\mbfc^{+i}=\mbfc^{\conv}$. We see that $i$ must be even. In fact, if $i$ is odd, then suppose $\mbfc=(c_1,c_2,\dots,c_m,c_1)$, then from $\mbfc^{+i}=\mbfc^{\conv}$, we have
    \[(c_{i+1},c_{i+2},c_{i+3},\dots,c_{i},c_{i+1})=(c_1,c_m,c_{m-1},\dots,c_2,c_1).\]
    This gives $c_{(m+i+1)/2}=c_{(m+i+3)/2}$, but this is impossible since $G$ is a simple graph without loops. Then by combining the relation \eqref{eq:cconv}, we have
    \[(\mbfc^{+i/2})^{\conv}=(\mbfc^{\conv})^{+(m-i/2)}=(\mbfc^{+i})^{+(m-i/2)}=\mbfc^{+i/2},\]
    and
    \[(\mbfc^{+(m+i)/2})^{\conv}=(\mbfc^{\conv})^{+(m-(m+i)/2)}=(\mbfc^{+i})^{+((m-i)/2)}=\mbfc^{+(m+i)/2}.\]
    If there exists another $k$ such that $(\mbfc^{+k})^{\conv}=\mbfc^{+k}$, then by the same argument as above, we have $\mbfc^{\conv}=\mbfc^{+2k}$. This gives $\mbfc^{+(2k-i)}=\mbfc$. Since $m$ is the smallest positive integer such that $\mbfc^{+m}=\mbfc$, so $2k-i$ is divisible by $m$. Thus $k=i/2$ or $(m+i)/2$ in the sense of modulo $m$.

\end{proof}

\begin{exa}
We provide an example to further elucidate this lemma in a more concrete manner. Let $G=C_4$ be $4$-cycle indexed clockwise  as $v_1, v_2,v_3$ and $v_4$. Let $\mbfc=(v_1,v_2,v_3,v_4,v_3,v_2,v_1,v_4,v_1)$ be a closed walk of length $m=8$, then for $i=3$, we have
\[\mbfc^{+3}=(v_4,v_3,v_2,v_1,v_4,v_1,v_2,v_3,v_4)=(\mbfc^{+3})^{\conv}.\]
Similarly, we also have $\mbfc^{+7}=(\mbfc^{+7})^{\conv}$.
\end{exa}


\begin{lem}\label{lem:wn2}
    Let $m$ be an even integer. Then
    \[
    |\{\mbfc \in \mathcal{C}(m) \, : \, \mbfc^{\conv} = \mbfc\}| = |\mathcal{W}(m/2)|.
    \]
\end{lem}

\begin{proof}
    Let $\mbfc=(c_1,c_2,\dots,c_m,c_1)$, since $\mbfc^{\conv}=\mbfc$, we see that $c_i=c_{m+2-i}$ for $i=2,3,\dots,m/2$. This shows that $\mbfc$ is totally determined by $c_1,c_2,\dots,c_{m/2+1}$. Then the result follows directly.
\end{proof}

\begin{lem}[\cite{WX1}, Lemma 4.2]\label{even1}
Let $G$ be a graph with adjacency matrix $A(G)$. Then $e^{\rm T}A(G)^me$ is even for any integer $m\geq 1$.
\end{lem}

\begin{lem}\label{lem:trabeven}
    Let $A$ and $B$ be two $n \times n$ symmetric integral matrices and the diagonal entries of $A$ are even. Then $\Tr(AB)$ is even.
\end{lem}

\begin{proof}
    Let $A=(a_{ij}),B=(b_{ij})$, then
    \[\Tr(AB)=\sum_{i=1}^n\sum_{j=1}^na_{ij}b_{ji}\equiv \sum_{i\neq j}a_{ij}b_{ji}=2\sum_{i<j}a_{ij}b_{ji}\equiv 0\pmod{2}.\]
\end{proof}

\begin{lem}\label{lem:a1a22}
    Let $A_1$ and $A_2$ be two symmetric integral matrices and $\ell$ be a positive even integer. Then we have $ e ^{\rm T}(A_1+2A_2)^{\ell} e \equiv  e ^{\rm T}A_1^{\ell} e  \pmod{4}$.
\end{lem}

\begin{proof}
    By expanding $(A_1+2A_2)^{\ell}$ directly, for each term, if it contains at least two $2A_2$, then it contributes $0\pmod{4}$. So it is enough to prove
    \begin{equation}\label{eq:A1A2}
         e ^{\rm T}\left(A_1^{\ell-1}A_2+A_1^{\ell-2}A_2A_1+\dots+A_2A_1^{\ell-1}\right) e \equiv 0\pmod{2}.
    \end{equation}
    Since $A_1$ and $A_2$ are symmetric, so for any $k$, $A_1^kA_2A_1^{\ell-k-1}+A_1^{\ell-k-1}A_2A_1^{k}$ is symmetric and $A_1^kA_2A_1^{\ell-k-1}-A_1^{\ell-k-1}A_2A_1^{k}$ is antisymmetric. This implies that the diagonal elements of $A_1^kA_2A_1^{\ell-k-1}+A_1^{\ell-k-1}A_2A_1^{k}$ are even, so the equality \eqref{eq:A1A2} follows immediately.
\end{proof}

When it comes to congruence relations between the traces of matrix powers, we must mention the Euler congruence \cite{Zarelua2008}, an important phenomenon in mathematics, which states that
\[
\Tr(A^{p^r}) \equiv \Tr(A^{p^{r-1}}) \pmod{p^r},
\]
for all integral matrices \( A \), all primes \( p \), and all \( r \in \mathbb{Z} \). The lemma we will prove below is related to this result but applies specifically to symmetric matrices with additional restrictions on their diagonal entries.

\begin{lem}\label{lem:trace2t}
    Let $A_1$ and $A_2$ be two symmetric integral matrices and the diagonal entries of $A_1$ are even numbers. Let $t\geq 2$ be a positive integer, then we have
    \begin{equation}\label{eq:expanofa}
        \Tr(A_1+2A_2)^{2^t}\equiv \Tr(A_1^{2^t})\pmod{2^{t+2}}.
    \end{equation}
\end{lem}

\begin{proof}

We distinguish the following two cases.

\

\noindent
\textbf{Case 1.} $t\geq 3$.

\

By expanding $(A_1+2A_2)^{2^t}$, we get
    \begin{equation}\label{eq11}(A_1+2A_2)^{2^t}=\sum_{B_1,B_2,\dots,B_{2^t}\in \{A_1,A_2\}}2^{|\{i\,|\,B_i=A_2\}|}B_1B_2\dotsm B_{2^t}.\end{equation}
    If $B_i=A_1$ for all $i$, then it contributes the term $A_1^{2^t}$. So we only need to consider the case $B_i=A_2$ for some $i$. Fix a product $B_1B_2\dotsm B_{2^t}$, we first assume that $A_2$ occurs at least three times in $B_1,B_2,\dots,B_{2^t}$. Let $k$ be the smallest positive integer such that $B_i=B_{i+k}$ for all $i$, where the index is understood in the sense $\pmod{2^t}$. It is easy to see that $k\,|\,2^t$, so we write $k=2^d$ for some $0\leq d\leq t$. Since $\Tr(AB)=\Tr(BA)$, we have
    \[\Tr(B_1B_2\dotsm B_{2^t})=\Tr(B_2B_3\dotsm B_{2^t}B_1)=\dots=\Tr(B_kB_{k+1}\dotsm B_{k-2}B_{k-1}).\]
    This implies that there are $k$ terms in the summation of \eqref{eq11} whose trace values are equal. By our assumption, there exists at least one number $1\leq i\leq k$ such that $B_i=A_2$, so among the terms $B_1,B_2,\dots,B_{2^t}$, $A_2$ occurs at least  $\max\{2^t/k,3\}$ times. By collecting like terms, we see that the coefficient of $\Tr(B_1B_2\dotsm B_{2^t})$ is divided by $k2^{\max\{2^t/k,3\}}=2^{d+\max\{2^{t-d},3\}}$. A direct analysis shows that $d+\max\{2^{t-d},3\}\geq t+2$ for all $d=0,1,\dots,t$. In this case, we see that the contribution of \( \Tr(B_1B_2\dotsm B_{2^t}) \) is $0$ by modulo $2^{t+2}$.

    Secondly, we assume that $A_2$ occurs exactly twice in $B_1,B_2,\dots,B_{2^t}$. Then $4\Tr(B_1B_2\dotsm B_{2^t})$ must equal to $4\Tr(A_1^iA_2A_1^jA_2)$ for some $i,j$ with $i+j=2^t-2$. If $i\neq j$, then
    \begin{align*} &\Tr(A_1^iA_2A_1^jA_2)=\Tr(A_1^{i-1}A_2A_1^jA_2A_1)=\dots=\Tr(A_2A_1^jA_2A_1^i)\\ =&\Tr(A_1^jA_2A_1^iA_2)=\Tr(A_1^{j-1}A_2A_1^iA_2A_1)=\dots=\Tr(A_2A_1^iA_2A_1^j),
    \end{align*} which implies that the term \( 4\Tr(A_1^i A_2 A_1^j A_2) \) appears \( 2^t \) times. Consequently, its contribution modulo \( 2^{t+2} \) is 0. If \( i = j = 2^{t-1} - 1 \), then, as in the argument above, this term appears \( 2^{t-1} \) times. Since \( A_2 A_1^{2^{t-1}-1} A_2 \) is symmetric and the diagonal entries of \( A_1^{2^{t-1}-1} \) are even, Lemma \ref{lem:trabeven} implies that \( \Tr(A_1^i A_2 A_1^j A_2) \) is even. Therefore, the contribution of \( \Tr(B_1B_2\dotsm B_{2^t}) \) modulo \( 2^{t+2} \) is always 0 as well.

Finally, we assume that \( A_2 \) appears exactly once among \( B_1, B_2, \dots, B_{2^t} \). In this case, \( 2\Tr(B_1 B_2 \dotsm B_{2^t}) \) is equal to \( 2\Tr(A_2 A_1^{2^t-1}) \). Similarly, we find that the term \( 2\Tr(A_2 A_1^{2^t-1}) \) appears \( 2^t \) times. By Lemma \ref{lem:2k1}, \( A_1^{2^t-1} \) is symmetric with all diagonal entries even, so Lemma \ref{lem:trabeven} implies that \( \Tr(A_2 A_1^{2^{t}-1}) \) is even. Therefore, the contribution of \( \Tr(B_1 B_2 \dotsm B_{2^t}) \) modulo \( 2^{t+2} \) is 0.

\

\noindent
\textbf{Case 2.} $t=2$.

\

By direct computation, we have
    \begin{align*}
    &\Tr(A_1+2A_2)^4-\Tr(A_1^4)\\
    =&8\Tr(A_1^3A_2)+16\Tr(A_1^2A_2^2)+8\Tr(A_1A_2A_1A_2)+32\Tr(A_1A_2^3)+16\Tr(A_2^4).
    \end{align*}
    Applying Lemma \ref{lem:trabeven}, we see that $\Tr(A_1^3A_2),\Tr(A_1A_2A_1A_2)$ are even. This completes the proof.
\end{proof}

\section{Proof of Theorem~\ref{equiv1}}\label{sec3}

In this section, we provide the proof of Theorem \ref{equiv1} by establishing a stronger result as follows.

\begin{thm}\label{thm:tracemod4}
    Let $A$ be the adjacency matrix of a graph $G$, $m=2^t(2k+1)$ be an integer with $k\in \mbn$. Then we have
    \[e^{\rm T}A^{2m} e - e^{\rm T}A^m e\equiv \frac{1}{2^{t+1}}\Tr\left(A^{4m}-A^{2m}\right)\pmod{4}.\]
\end{thm}

\begin{proof}
    We first prove the result for $m=2^t$. Let $\mbfc\in \mathcal{C}(m)$. Let $d$ be the smallest positive integer such that $\mbfc^{+d}=\mbfc$, then $\mbfc,\mbfc^{+1},\dots,\mbfc^{+(d-1)}$ are all distinct, so the cardinality of $\bar{\mbfc}$ is $d$. Suppose $m=qd+r$ with $0\leq r<d$, then we have $\mbfc^{+r}=(\mbfc^{+qd})^{+r}=\mbfc^{+m}=\mbfc$. The minimality of $d$ forces that $r = 0$, i.e. $d\mid m$. Therefore, the closed walk $\mbfc$ is totally determined by $c_1,\ldots,c_d$. This implies that
    \[|\{\mbfc\in \mathcal{C}(m)\,:|\bar{\mbfc}|=d\}|=|\{\mbfc\in \mathcal{C}(d)\,:|\bar{\mbfc}|=d\}|,\]
    which gives
    \begin{align*}
        \Tr(A^m)=|\mathcal{C}(m)|=\sum_{d\,|\, m}|\{\mbfc\in \mathcal{C}(m)\,:|\bar{\mbfc}|=d\}=\sum_{d\,|\, m}|\{\mbfc\in \mathcal{C}(d)\,:|\bar{\mbfc}|=d\}.
    \end{align*}
    Since $m$ is a power of $2$, we get
    \[\frac{1}{2m}\Tr\left(A^{4m}-A^{2m}\right)=\frac{1}{2m}|\{\mbfc\in \mathcal{C}(4m)\,:|\bar{\mbfc}|=4m\}|.\]

    Now for each $\mbfc\in \mathcal{C}(4m)$ with $|\bar{\mbfc}|=4m$, if $\bar{\mbfc}\neq \ol{\mbfc^{\conv}}$, then we show that $\bar{\mbfc}\cap \ol{\mbfc^{\conv}}=\emptyset$. Otherwise, if $\mbfc^{+k}\in \bar{\mbfc}\cap \ol{\mbfc^{\conv}}$ for some $k\in \mbz$, then $\mbfc^{+k}=(\mbfc^{\conv})^{+\ell}$ for some $\ell\in \mbz$. For any $d\in \mbz$, we have
    \[\mbfc^{+d}=(\mbfc^{+k})^{+(d-k)}=(\mbfc^{\conv})^{+(\ell+d-k)}\in \ol{\mbfc^{\conv}}.\]
    This implies that $\bar{\mbfc}\subseteq \ol{\mbfc^{\conv}}$. Similarly, we also have $ \ol{\mbfc^{\conv}}\subseteq \bar{\mbfc}$, which contradicts to our assumption $\bar{\mbfc}\neq\ol{\mbfc^{\conv}}$. So in this case, $\bar{\mbfc}$ and $\ol{\mbfc^{\conv}}$ are both contained in the set $\{\mbfc\in \mathcal{C}(4m)\,:|\bar{\mbfc}|=4m\}$ and $|\bar{\mbfc}\cup\ol{\mbfc^{\conv}}|=8m$. After dividing by $2m$, they have no contribution modulo $4$. So we get
    \[\frac{1}{2m}\Tr\left(A^{4m}-A^{2m}\right)\equiv \frac{1}{2m}|\{\mbfc\in \mathcal{C}(4m)\,:|\bar{\mbfc}|=4m,\bar{\mbfc}=\ol{\mbfc^{\conv}}\}|\pmod{4}.\]
    By Lemma \ref{lem:akconv}, there are exactly two representatives $\mbfb\in \bar{\mbfc}$ such that $\mbfb=\mbfb^{\conv}$, so we get
    \[\frac{1}{2m}\Tr\left(A^{4m}-A^{2m}\right)\equiv |\{\mbfc\in \mathcal{C}(4m)\,:|\bar{\mbfc}|=4m,\mbfc=\mbfc^{\conv}\}|\pmod{4}.\]
    Since $|\bar{\mbfc}|=4m$, we know $\mbfc^{+2m}\neq \mbfc$. Together with Lemma \ref{lem:wn2}, we get
    \begin{align*}
        &|\{\mbfc\in \mathcal{C}(4m)\,:|\bar{\mbfc}|=4m,\mbfc=\mbfc^{\conv}\}|\\
        =&|\{\mbfc\in \mathcal{C}(4m)\,:\mbfc=\mbfc^{\conv}\}|-|\{\mbfc\in \mathcal{C}(4m)\,:|\bar{\mbfc}|<4m,\mbfc=\mbfc^{\conv}\}|\\
        =&|\{\mbfc\in \mathcal{C}(4m)\,:\mbfc=\mbfc^{\conv}\}|-|\{\mbfc\in \mathcal{C}(4m)\,:\mbfc^{+2m}=\mbfc=\mbfc^{\conv}\}|\\
        =&|\mathcal{W}(2m)|-|\mathcal{W}(m)|= e ^{\rm T}A^{2m} e  - e ^{\rm T}A^m e .
    \end{align*}
This gives the desired result for $m=2^t$.

Now we consider the case $m=2^t(2k+1)$. By Lemma \ref{lem:2k1}, we can express $A^{2k+1}$ as the sum of two matrices, $A_1$ and $2A_2$, where  $A_1$ is the adjacency matrix of a specific simple graph, and $A_2$ is an integral matrix. When $t\geq 1$, by Lemma \ref{lem:a1a22}, we have
\[ e ^{\rm T}A^{2m} e - e ^{\rm T}A^m e \equiv e ^{\rm T}A_1^{2^{t+1}} e - e ^{\rm T}A_1^{2^t} e  \pmod{4}.\]
On the other hand, by Lemma \ref{lem:trace2t}, we have
\[\frac{1}{2^{t+1}}\Tr(A_1+2A_2)^{2^{t+2}}-\frac{1}{2^{t+1}}\Tr(A_1+2A_2)^{2^{t+1}}\equiv \frac{1}{2^{t+1}}\Tr (A_1^{2^{t+2}})-\frac{1}{2^{t+1}}\Tr (A_1^{2^{t+1}})\pmod{4}.\]
Hence the result follows immediately from the case $m=2^t$.

Finally, we consider the case $t=0$. By Lemma \ref{lem:a1a22}, we have
\begin{equation}\label{eq:eta2}
     e ^{\rm T}A^{2m} e - e ^{\rm T}A^m e \equiv e ^{\rm T}A_1^{2} e - e ^{\rm T}(A_1+2A_2) e  \equiv  e ^{\rm T}A_1^{2} e - e ^{\rm T}A_1 e -2\Tr (A_2)\pmod{4},
\end{equation}
where the last congruence follows from $ e ^{\rm T}A_2 e  \equiv \Tr (A_2)\pmod{2}$ since $A_2$ is symmetric. On the other hand, we have
\begin{equation}\label{eq:trt1}
\begin{split}
    &\frac{1}{2}\Tr(A_1+2A_2)^4-\frac{1}{2}\Tr(A_1+2A_2)^2\\
    \equiv &\frac{1}{2}\Tr (A_1^4)-\frac{1}{2}\Tr (A_1^2)+2\Tr(A_1A_2)-2\Tr (A_2^2)\pmod{4}.
\end{split}
\end{equation}
By Lemma \ref{lem:trabeven}, we have $\Tr(A_1A_2)\equiv 0\pmod{2}$. Since $A_2$ is symmetric, we have $\Tr(A_2^2)\equiv \Tr(A_2)\pmod{2}$. Combining the equations \eqref{eq:eta2} and \eqref{eq:trt1}, we get the desired result.
\end{proof}

\begin{rmk}
    It is not difficult to see that Theorem \ref{thm:tracemod4} holds for symmetric integral matrices with even diagonal entries.
\end{rmk}

Additionally, we recall the well-known result that the similarity of matrices is determined by the trace of its powers.

\begin{lem}\label{lem:trace}
$\phi(G;x)=\phi(H;x)$ if and only if $\Tr(A(G)^m)=\Tr(A(H)^m)$ for any integer $m\geq 0$.
\end{lem}

\begin{proof}
    Suppose $\phi(G;x) = \phi(H;x)$. Then $A(G)$ and $A(H)$ have the same eigenvalues, denoted by $\lambda_1, \dots, \lambda_n$. Therefore,
    \[
    \Tr(A(G)^m) = \lambda_1^m + \lambda_2^m + \dots + \lambda_n^m = \Tr(A(H)^m).
    \] Conversely, assume that $\Tr(A(G)^m) = \Tr(A(H)^m)$ holds for all integers $m \geq 0$. Let $\lambda_1, \dots, \lambda_r$ denote the distinct eigenvalues of $A(G)$ with multiplicities $k_1, \dots, k_r$, and let $\mu_1, \dots, \mu_s$ be the distinct eigenvalues of $A(H)$ with multiplicities $\ell_1, \dots, \ell_s$. Then, for any \( m \geq 0 \), we have
    \[
    k_1 \lambda_1^m + \dots + k_r \lambda_r^m = \ell_1 \mu_1^m + \dots + \ell_s \mu_s^m.
    \] Now it suffices to show that, under a suitable ordering, $\lambda_i = \mu_i$ and $k_i = \ell_i$ for all \( i \). If it is not the case, then by collecting like terms, we may assume that there exist $n_1,\dots,n_t \neq 0$ with $t \geq 1$ such that
    \[
    n_1 \nu_1^m + \dots + n_t \nu_t^m = 0.
    \]
    holds for all \( m \geq 0 \), where $\nu_1, \dots, \nu_t$ are the distinct values among $\lambda_i$ and $\mu_j$. However, since the Vandermonde determinant of the distinct values \( \nu_1, \dots, \nu_t \) is non-zero, this implies that \( n_1 = \dots = n_t = 0 \), which is a contradiction.
\end{proof}

Moreover, for a general positive integer \( m \), to obtain an explicit formula for \( e^{\rm T}A(G)^m e \) modulo 4 in terms of \( \Tr(A(G)^k) \), we need to prove the following lemma.

\begin{lem}\label{lem:oddm}
Let $A$ be the adjacency matrix of a graph $G$ and $m$ be an odd integer. Then we have
\[e^{\rm T}A^me\equiv \Tr(A^{2m}+A^m)\pmod{4}.\]
\end{lem}

\begin{proof}
We first assume that $m=1$. Then $\Tr(A)=0$ and
\[\Tr(A^2)=\sum_{i,j}a_{ij}^2=\sum_{i,j}a_{ij}=e^{\rm T}Ae.\]
For $m > 1$, by Lemma \ref{lem:2k1}, we can express $A^{2k+1}$ as the sum of two matrices, $A_1$ and $2A_2$, where  $A_1$ is the adjacency matrix of a specific simple graph, and $A_2$ is an integral matrix. By the discussion above, we have $e^{\rm T}A_1e\equiv \Tr(A_1^2+A_1)\pmod{4}$. On the other hand, we have
\[\Tr(A^{2m})=\Tr(A_1^2+2A_1A_2+2A_2A_1+4A_2^2)\equiv \Tr(A_1^2)\pmod{4}.\]
Finally, since $A_2$ is symmetric, we have $e^{\rm T}A_2e\equiv \Tr(A_2)\pmod{2}$. Combining all the equalities, we have
\[e^{\rm T}A^me=e^{\rm T}(A_1+2A_2)e\equiv \Tr(A_1^2+A_1)+2\Tr(A_2)\equiv \Tr(A^{2m}+A^m)\pmod{4}.\]
\end{proof}

Now we can obtain the following explicit formula for \( e^{\rm T}A(G)^m e \) modulo 4 in terms of \( \Tr(A(G)^k) \).

\begin{prop}\label{prop:tracemod4}
Let $A$ be the adjacency matrix of a simple graph $G$, $m=2^t(2k+1)$ be an integer with $k\in \mbn$. Then we have
\[e^{\rm T}A^me\equiv \frac{1}{2^{t+1}}\Tr\left(A^{2m}\right)+\sum_{\ell=0}^{t+1}\frac{1}{2^{\ell}}\Tr\left(A^{2^{\ell}(2k+1)}\right)\pmod{4}.\]
\end{prop}

\begin{proof}
It follows directly by Theorem \ref{thm:tracemod4} and Lemma \ref{lem:oddm}. \end{proof}

We are now in a position to deduce Theorem \ref{equiv1} from Proposition \ref{prop:tracemod4}.

\begin{proof}[Proof of Theorem \ref{equiv1}] Let $m=2^t(2k+1)$ with $k\in \mbn$. By Proposition \ref{prop:tracemod4}, we get
\[e^{\rm T}A(G)^me\equiv \frac{1}{2^{t+1}}\Tr\left(A(G)^{2m}\right)+\sum_{\ell=0}^{t+1}\frac{1}{2^{\ell}}\Tr\left(A(G)^{2^{\ell}(2k+1)}\right)\pmod{4},\]
and
\[e^{\rm T}A(H)^me\equiv \frac{1}{2^{t+1}}\Tr\left(A(H)^{2m}\right)+\sum_{\ell=0}^{t+1}\frac{1}{2^{\ell}}\Tr\left(A(H)^{2^{\ell}(2k+1)}\right)\pmod{4}.\]
Since $G$ and $H$ are cospectral, it follows from Lemma \ref{lem:trace} that for any $m\geq 0$, we have
$e^{\rm T}A(G)^me\equiv e^{\rm T}A(H)^me\pmod{4}$.
\end{proof}

\begin{rmk}
In some studies (see e.g. \cite{Merikoski1984, Marcus1962}), researchers have also examined the trace and the sum of elements of the powers of a matrix. However, most of these works focus on inequalities and asymptotic properties, as finding an exact relationship between them is nearly impossible. Therefore, our Proposition \ref{prop:tracemod4} offers a noteworthy result.
\end{rmk}

\section{Proofs of Theorems~\ref{parity11} and~\ref{equiv3}}\label{sec4}

In this section, we give the proofs of Theorem \ref{parity11} and Theorem~\ref{equiv3}.

\begin{proof}[Proof of Theorem \ref{parity11}]
By Lemma~\ref{even1}, both $\frac{e^{\rm T}A(G)^me}2$ and $\frac{e^{\rm T}A(H)^me}2$ are integers for any $m\geq 1$, so Theorem \ref{equiv1} implies that $\frac{e^{\rm T}A(G)^me}2\equiv \frac{e^{\rm T}A(H)^me}2\pmod{2}$ for any $m\geq 1$.

We distinguish the following two cases.

\

\noindent
\textbf{Case 1.} $n$ is even.

\

Note that $W(G)^{\rm T}W(G)=(e^{\rm T}A(G)^{i+j-2}e)_{n\times n}$ and $W(H)^{\rm T}W(H)=(e^{\rm T}A(H)^{i+j-2}e)_{n\times n}$.  Therefore, we have
\[\frac{W(G)^{\rm T}W(G)}2\equiv \frac{W(H)^{\rm T}W(H)}2\pmod{2}.\]
Note that
 \[\det\left(\frac{W(G)^{\rm T}W(G)}2\right)=(2^{-n/2}\det W(G))^2=\eta(G)^2,\]
and
\[\det\left(\frac{W(H)^{\rm T}W(H)}2\right)=(2^{-n/2}\det W(H))^2=\eta(H)^2.\]
If $\eta(G)$ is odd, then $2^{-n/2}\det W(G)$ is odd and hence $\det(\frac{W(H)^{\rm T}W(H)}2)$ is odd, which implies that $\eta(H)=2^{-n/2}|\det W(H)|$ is odd. Similarly, if $\eta(H)$ is odd, then $\eta(G)$ is also odd.

\

\noindent
\textbf{Case 2.} $n$ is odd.

\

Let $W_1(G)$ (resp. $W_1(H)$) be the matrix obtained from $W(G)$ (resp. $W(H)$) by
replacing the first column $e$ with $2e$, i.e.
\[W_1(G) := [2e, A(G)e,\ldots , A(G)^{n-1}e]~{\rm  and}~ W_1(H) := [2e, A(H)e, \ldots , A(H)^{n-1}e].\]
Notice that both $\frac{W_1(G)^{\rm T}W(G)}2$ and $\frac{W_1(H)^{\rm T}W(H)}2$ are integral matrices. It follows that
\[\frac{W_1(G)^{\rm T}W(G)}2\equiv \frac{W_1(H)^{\rm T}W(H)}2\pmod{2}.\]
It is easy to verify that
\[\det\left(\frac{W_1(G)^{\rm T}W(G)}2\right)=(2^{-(n-1)/2}\det W(G))^2=\eta(G)^2,\]
and
\[\det\left(\frac{W_1(H)^{\rm T}W(H)}2\right)=(2^{-(n-1)/2}\det W(H))^2=\eta(H)^2.\]
Then using the same arguments as Case 1, we can show that $\eta(H)$ is odd if and only if $\eta(G)$ is odd.

Combining Cases 1 and 2, we conclude that \( \eta(H) \) is odd if and only if \( \eta(G) \) is odd, that is, \( \eta(G) \) and \( \eta(H) \) have the same parity.
\end{proof}

Next, we present the proof of Theorem \ref{equiv3}, which gives the third invariant for cospectral graphs.

For integers $m \geq 0$, let $N_m(G)=e^{\rm T}A(G)^me$ be the number of walks of length $m$ in $G$. Then the generating series of $N_m(G)$ can be obtained from the characteristic polynomial of $G$ and that of the complement $\bar{G}$ of $G$.

\begin{lem}[\cite{CDS}, Theorem 1.11]\label{lem:hgt}
    Let $G$ be a graph with complement $\Bar{G}$, and let $H_G(t)=\sum_{m=0}^{\infty}N_m(G)t^m$ be the generating series of the numbers $N_m(G)\ (m=0,1,2,\cdots)$. Then \[H_G(t)=\frac{1}{t}\left((-1)^n\frac{\phi(\Bar{G};-\frac{t+1}{t})}{\phi(G;\frac{1}{t})}-1\right).\]
\end{lem}

Now we give an equivalent formulation of Theorem~\ref{equiv1}.

\begin{thm}\label{equiv}
Let $G$ and $H$ be two cospectral graphs. Then the following two statements are equivalent:
\begin{enumerate}
\item $e^{\rm T}A(G)^me\equiv e^{\rm T}A(H)^me\pmod{4}$, for any $m\geq 0$;
\item $\phi(\bar{G};x)\equiv \phi(\bar{H};x)\pmod{4}$.
\end{enumerate}
\end{thm}

\begin{proof}
By Lemma \ref{lem:hgt} we have
\begin{equation}\label{eq:genseries}
t^{n+1}\phi\left(G;\frac{1}{t}\right)\sum_{m\geq 0}N_m(G)t^m=(-1)^nt^n\phi\left(\bar{G};-\frac{t+1}{t}\right)-t^n\phi\left(G;\frac{1}{t}\right),
\end{equation}
where the equality holds as a formal power series in $t$. It is easy to see that the coefficients on both hands are integers.

If $N_m(G)\equiv N_m(H)\pmod{4}$ for all $m\geq 0$, then from $\phi(G;x)=\phi(H;x)$, we get
\[t^n\phi\left(\bar{G};-\frac{t+1}{t}\right)\equiv t^n\phi\left(\bar{H};-\frac{t+1}{t}\right)\pmod{4}.\]
By changing the variables, we see that $\phi(\bar{G};x)\equiv \phi(\bar{H};x)\pmod{4}$.

Conversely, if $\phi(\bar{G};x)\equiv \phi(\bar{H};x)\pmod{4}$, then from Equation \eqref{eq:genseries}, we see get
\begin{equation}\label{eq:nkgh}
t^n\phi(G;1/t)\sum_{m\geq 0}(N_m(G)-N_m(H))t^m\equiv 0\pmod{4}.
\end{equation}
Note that \( t^n \phi(G; 1/t) \) is a polynomial with a constant term of 1, which implies that \( t^n \phi(G; 1/t) \) is invertible in the formal power series ring \( \mathbb{Z}/4\mathbb{Z}[\![t]\!] \). Consequently, we have \( \sum_{m \geq 0} (N_m(G) - N_m(H)) t^m \equiv 0 \pmod{4} \), and therefore \( N_m(G) \equiv N_m(H) \pmod{4} \) for all \( m \geq 0 \).
\end{proof}

Now we are ready to present the following.

\begin{proof}[Proof of Theorem \ref{equiv3}]
This follows immediately from Theorem~\ref{equiv} and Theorem~\ref{equiv1}.
\end{proof}

Finally, we end this paper by giving an example for illustrations, see~\cite{WYZ}.}

\begin{exa}Let $G$ and $H$ be two graphs on $n=9$ vertices with adjacency matrices $A$ and $B$ given as follows, respectively.
\[A=
\begin{pmatrix}
 0 & 1 & 0 & 1 & 0 & 0 & 0 & 0 & 1 \\
 1 & 0 & 1 & 1 & 1 & 0 & 1 & 0 & 0 \\
 0 & 1 & 0 & 1 & 0 & 0 & 1 & 0 & 0 \\
 1 & 1 & 1 & 0 & 1 & 1 & 0 & 0 & 0 \\
 0 & 1 & 0 & 1 & 0 & 1 & 0 & 0 & 0 \\
 0 & 0 & 0 & 1 & 1 & 0 & 1 & 0 & 0 \\
 0 & 1 & 1 & 0 & 0 & 1 & 0 & 1 & 0 \\
 0 & 0 & 0 & 0 & 0 & 0 & 1 & 0 & 0 \\
 1 & 0 & 0 & 0 & 0 & 0 & 0 & 0 & 0 \\
\end{pmatrix},\quad B=\begin{pmatrix}
 0 & 1 & 1 & 0 & 1 & 1 & 0 & 0 & 1 \\
 1 & 0 & 1 & 1 & 0 & 0 & 0 & 0 & 0 \\
 1 & 1 & 0 & 1 & 1 & 0 & 1 & 0 & 0 \\
 0 & 1 & 1 & 0 & 1 & 0 & 0 & 0 & 0 \\
 1 & 0 & 1 & 1 & 0 & 1 & 0 & 0 & 0 \\
 1 & 0 & 0 & 0 & 1 & 0 & 1 & 0 & 0 \\
 0 & 0 & 1 & 0 & 0 & 1 & 0 & 1 & 0 \\
 0 & 0 & 0 & 0 & 0 & 0 & 1 & 0 & 0 \\
 1 & 0 & 0 & 0 & 0 & 0 & 0 & 0 & 0 \\
\end{pmatrix}
.\]

It can be computed
\[\phi(G;x)=\phi(H;x)=-2 - 4 x + 16 x^2 + 27 x^3 - 24 x^4 - 37 x^5 + 10 x^6 + 14 x^7 - x^9,\]
and thus $G$ and $H$ are cospectral. Also we have
\[\phi(\bar{G};x)=4 + 31 x + 66 x^2 - 3 x^3 - 108 x^4 - 43 x^5 + 38 x^6 + 22 x^7 - x^9,\]
and
\[\phi(\bar{H};x)=4 + 35 x + 74 x^2 - 7 x^3 - 112 x^4 - 43 x^5 + 38 x^6 + 22 x^7 - x^9.\] It is easy to see that $\phi(\bar{G};x)\equiv \phi(\bar{H};x)\pmod{4}$ holds. Moreover, we have
\[\eta(G)=2^{-\lfloor n/2 \rfloor}|\det W(G)|=1 \text{ and } \eta(H)=2^{-\lfloor n/2 \rfloor}|\det W(H)|=587.\]
Thus, $\eta(G)\equiv~\eta(H)\pmod{2}$. For $0\leq m\leq 9$, the sequences of $\{e^{\rm T}A(G)^me\}_{0\leq m\leq 9}$ and $\{e^{\rm T}A(H)^me\}_{0\leq m\leq 9}$ are
\[(9,28, 104, 380, 1412, 5210, 19308, 71376, 264260, 977480),\] and \[(9,28, 104, 380, 1408, 5198, 19248, 71176, 263452, 974620),\] respectively.
Thus $e^{\rm T}A(G)^me\equiv e^{\rm T}A(H)^me\pmod{4}$ for $0\leq m\leq 9$, and it also holds for any $m\geq 10$ by the Cayley-Hamilton Theorem.
\end{exa}

\section*{Acknowledgments}
The research of the third author is supported by National Key Research and Development Program of China 2023YFA1010203 and National Natural Science Foundation of China (Grant No.\,12371357), and the fourth author is supported by Fundamental Research Funds for the Central Universities (Grant No.\,531118010622), National
Natural Science Foundation of China (Grant No.\,1240011979) and Hunan Provincial Natural Science Foundation of China (Grant No.\,2024JJ6120).

\end{document}